\documentclass[12pt]{article}
\usepackage[english]{babel}
 \usepackage{graphicx,color}
\usepackage{amsfonts}
\usepackage{amssymb,amsmath,amsthm,scalefnt,mathrsfs}
\usepackage{fancyhdr}
\usepackage[all]{xy}

\newtheorem{theo}{Theorem}[section]
\newtheorem{cor}[theo]{Corollary}
\newtheorem{lem}[theo]{Lemma}

\newtheorem{defn}[theo]{Definition}

\newcommand{\set}[1]{\left\{#1\right\}}

\newcommand{\Z}{\mathbb{Z}}
\newcommand{\N}{\mathbb{N}}
\newcommand{\NZ}{\mathbb{S}}

\newcommand{\s}{\sigma}

\newcommand{\AS}{\mathcal{A}^\mathbb{S}}

\newcommand{\Alf}{\mathcal{A}}

\newcommand{\x}{\mathbf{x}}
\newcommand{\y}{\mathbf{y}}

\newcommand{\f}{\mathbf{f}}

\newcommand{\cqd}{\hspace{11.5cm}$\square$ }
\newcommand{\AMSclassification}[1]{\par\addvspace\baselineskip\noindent\textbf{Mathematical subject classification:}\enspace\ignorespaces#1}
\newcommand{\keywords}[1]{\par\addvspace\baselineskip\noindent\textbf{Keywords:}\enspace\ignorespaces#1}

\title{Standard decomposition of expansive ergodically supported dynamics}

\author{Marcelo Sobottka\footnote{E-mail address: sobottka@mtm.ufsc.br}}

\date{}

\begin{document}
\maketitle

\footnotesize{\centerline{Departamento de
Matem\'atica,}\centerline{Centro de Ci\^{e}ncias F\'{\i}sicas e
Matem\'{a}ticas,} \centerline{Universidade de Federal de Santa Catarina}
  \centerline{CEP 88040-900,
Florian\'{o}polis - SC, Brazil.}}\normalsize

\begin{abstract}
In this work we introduce the notion of weak quasigroups, that are quasigroup operations defined almost everywhere on some set. Then we prove that the topological entropy and the ergodic period of
an invertible expansive ergodically supported dynamical system $(X,T)$ with the shadowing property establishes a sufficient criterion for
the existence of quasigroup operations defined almost everywhere outside of universally null sets and for which $T$ is an automorphism. Furthermore, we find a decomposition of the dynamics of $T$ in terms of $T$-invariant weak topological subquasigroups.
\end{abstract}

\footnotesize

\AMSclassification{68P30, 17C10, 20N05.}

\normalsize

\bigskip
\hrule
\noindent
{\footnotesize\em This is a pre-copy-editing, author-produced preprint of an article accepted for publication in Nonlinear Dynamics, ISSN 0924-090X. The definitive publisher-authenticated version is available online at:\newline
http://dx.doi.org/10.1007/s11071-014-1383-4 or\newline
http://link.springer.com/article/10.1007\%2Fs11071-014-1383-4 .}
\hrule
\bigskip

\section{Introduction}\label{introduction}


The problem of characterizing the dynamical behavior of maps which
are endomorphisms for compact groups has been widely studied in the
last years (see for instance \cite{CeccheriniCoornaert,HostMaassMartinez,mmpy,mms,Melbourne,Miles,Pivato03,Quian_et_al,Sobottka2007,Sobottka2008}).

One of the first works on this subject is due to R.
Bowen \cite{bowen}, who studied the entropy of such maps and showed
that the Haar measure is the maximum entropy measure for a certain class
of algebraic dynamical systems. Later, in \cite{lind77}, D. Lind
proved that ergodic maps which are automorphisms for compact Abelian
groups are always conjugated to some full shift. For the case where
$(X,+)$ is any topological group with $X$ being a zero-dimensional
space, B. Kitchens \cite{kitchens} proved that any expansive endomorphism $T:X\to X$
can be represented as a shift map defined on the Cartesian product
of a full shift with a finite set. In \cite{marcus}, Sindhushayana et Al. proved an analogous result for shift spaces $X$ on some alphabet $A$ for which there exist a group $P(A)$ of permutations of
the elements of $A$ and a group shift $Y\subseteq P(A)^\Z$, such that $X$ is
invariant under the action of any element of $Y$. In \cite{Sobottka2007}, this
result was extended for expansive maps which are endomorphisms for
a certain class of zero-dimensional quasigroups.

In this work, we prove that the topological entropy and the ergodic period of
an invertible expansive ergodically supported dynamical system $(X,T)$ with the shadowing property provides a sufficient criterion for
the existence of quasigroup operations defined almost everywhere outside of universally null sets and for which $T$ is an automorphism (Theorem \ref{entropia->operacao}). As a consequence of this result, we prove that if $(X,T)$ is ergodically aperiodic and has topological entropy $\log(N)$ for an integer $N\notin\{2,8\}\cup\{2p:\ p\text{ is prime}\}$, then we can find a quasigroup operation defined almost everywhere and decompose the dynamics of $T$ in terms of a finite family of subquasigroups (Theorem \ref{standard_decomposition}). In this way we obtain for ergodic maps an analogue to the decomposition of linear maps in terms of their eigenspaces.

\bigskip

We say $(X,T)$ is a {\it topological dynamical system} if $X$ is a
compact metric space and $T:X\to X$ is a continuous onto map. The topological entropy of $(X,T)$ will be denoted by $\mathbf{h}(T)$.
We say $(X,T)$ and
$(Y,S)$ are {\em conjugated} if there
exists an invertible map $\f:X\to Y$ such that $\f\circ T=S\circ\f$. In the case when $\f$ is a homeomorphism we say $(X,T)$ and
$(Y,S)$ are {\em topologically conjugated}.

Given a finite alphabet $\Alf$, define $\AS:=\set{(x_i)_{i\in\NZ}: x_i\in\Alf,\forall
i\in\NZ}$, with $\NZ=\Z$ or $\NZ=\N$. We consider in $\AS$
the product topology which is generated by the clopen
subsets called cylinders. Let $\s_{\AS}:\AS\to\AS$ be the {\em shift map} defined by
$\s_{\AS}\bigl((x_i)_{i\in\NZ}\bigr)=(x_{i+1})_{i\in\NZ}$.
Therefore, a symbolic dynamical system is a topological dynamical
system $(\Lambda,\s_\Lambda)$ where $\Lambda\subseteq\AS$ is a closed subset such that $\s_{\AS}(\Lambda)=\Lambda$,
and $\s_\Lambda$ is the restriction
of $\s_{\AS}$ to $\Lambda$ (in this case we refer to $\Lambda$ as a {\em shift space}). A special type of shift spaces are the {\em Markov shifts}, which are those symbolic dynamical systems that can be constructed from walks on finite directed graphs (see \cite{LindMarcus} for more details).

A topological dynamical system $(X,T)$, is said to be {\em expansive}\footnote{The standard definition of expansiveness states that $(X,T)$ is expansive if there exists $\delta>0$ such that if $x\neq y$, then $d(T^n(\x),T^n(\y))>\delta$ for some $n\in \NZ$. Note that the standard definition implies the definition of expansiveness that we are adopting here.}
if there exists a family $\{U_i\}_{1\leq i\leq k}$ of open sets,
such that $\overline{\cup_{1\leq i\leq k} U_i}=X$ and for $\x,\y\in
X$, $\x\neq \y$, there exists $1\leq i\leq k$ and $n\in\NZ$ (with
$\NZ=\Z$ if $T$ is invertible and $\NZ=\N$ otherwise) so that
$T^n(\x)\in \bar{U}_i$ and $T^n(\y)\notin \bar{U}_i$. When $(X,T)$ is expansive we can define its symbolic
representation as the shift space $(\Lambda,\s_\Lambda)$, where
$\Lambda\subseteq\{1,\ldots,k\}^\NZ$ is such that
$(q_i)_{i\in\NZ}\in\Lambda$ if and only if there exists $\x_0\in X$
such that for all $i\in\NZ$ we have $T^i(\x_0)\in U_{q_i}$.

We say $(X,T)$ has the {\em shadowing property} if for any
$\epsilon>0$ there exists $\delta>0$ such that if
$(\y_n)_{n\in\NZ}\subset X$ is a sequence which verifies
$d(T(\y_n),\y_{n+1})<\delta$ for all $n\in\NZ$, then there exists
$\x_0\in X$ such that $d(T^n(\x_0),\y_n)<\epsilon$ for all $n\in\NZ$.

Given $(X,T)$, define $X^T:=\{(x_i)_{i\in\Z}:\ x_{i+1}=T(x_i), \forall i\in\Z\}$. It is well known that there exists a product metric on $X^T$ which makes it compact, and for which the shift map $\s_T:X^T\to X^T$ is continuous (see \cite[Chap. 5]{walters}). The topological dynamical system $(X^T,\s_T)$ is called the {\em inverse limit system} of $(X,T)$. The projection $p:X^T\to X$ which takes the sequence $(x_i)_{i\in\Z}$ to $x_0$ is continuous and commutes with the maps $\s_T$ and $T$, that is, $p\circ\s_T=T\circ p$. In fact, if $T$ is invertible, then for each $a\in X$ the unique sequence in $X^T$ with $x_0=a$
is $(T^i(a))_{i\in\Z}$, which means that $p$ is invertible and, since $p$ is a continuous function between compact spaces, it implies that $p^{-1}$ is also continuous.
Therefore, in such a case, $p$ is a topological conjugacy between $(X,T)$ and $(X^T,\s_T)$.

A probability measure on $X$ is said to be an {\em ergodically supported measure} if it is an ergodic measure which
assigns positive measure for any nonempty open subset of $X$. Thus, we say that $(X,T)$ is {\em ergodically supported} if there exists an
ergodically supported measure for it.  A set $E\subset X$ which has zero measure for any ergodically supported measure is said to be a {\em universally null set}.

If for all $n\geq 1$ we have that $(X,T^n)$ is ergodically
supported, then we say that $(X,T)$ is {\em ergodically aperiodic}. On
the other hand, we say that $(X,T)$ has {\em ergodic period $B\in\N$} if
$(X,T)$ is ergodically supported and there exists a finite family
$\{C_i\}_{0\leq i\leq B-1}$ of closed sets, such that:
$X=\bigcup_{i=0}^{B-1}C_i$; $C_i\cap C_j$ is a universally null set for any $i\neq j$;
$T(C_i)=C_{i+1\ (mod\ B)}$; and $(C_i,T^B)$ is ergodically aperiodic for all $i$.

An important concept in dynamical systems is the {\em almost topological conjugacy} of two dynamical systems. Such a concept was introduced by R. Adler and B. Marcus in \cite{adler} to study invariants of Markov shifts and later extended by W. Sun in \cite{sun} to dynamical systems whose symbolic representations are Markov shifts. Due to the central role played by almost topological conjugacies in this work, we present its definition and the main result about almost topological conjugate dynamical systems due to Sun:

\begin{defn}[Def. 1.1 in \cite{sun}] Two ergodically supported topological dynamical systems $(X,T)$ and $(Y,S)$ are said to be {\em almost topologically
conjugate} if there exist an ergodically supported
Markov shift $(\Lambda,\s_\Lambda)$ and two continuous onto maps
$\f_T:\Lambda\to X^T$ and $\f_S:\Lambda\to Y^S$ such that:\\

\noindent (i) $\sigma_{X^T}\circ\f_T=\f_T\circ\s_\Lambda$ and
$\sigma_{Y^S}\circ\f_S=\f_S\circ\s_\Lambda$;\\

\noindent (ii) There exist a $\sigma_{X^T}$-invariant universally null set $M_2\subset X^T$ and a $\sigma_{Y^S}$-invariant universally null set $P_2\subset Y^S$, such that $\f_T:\Lambda\setminus M_1\to X^T\setminus M_2$ and $\f_S:\Lambda\setminus P_1\to Y^S\setminus P_2$ are one-to-one, where $M_1=\f_T^{-1}(M_2)$ and $P_1=\f_S^{-1}(P_2)$.
\end{defn}

\begin{theo}[Theo. 1.2 in \cite{sun}]\label{SunMainTheo}
Let $(X,T)$ and $(Y,S)$ be ergodically supported expansive maps with the
shadowing property. Then $(X,T)$ and $(Y,S)$ are almost topologically
conjugate if and only if they have equal topological entropy
and the same ergodic period.
\end{theo}

\section{Quasigroups and weak quasigroups}\label{quasigroups}

Let $G$ be a set and let $*$ be a binary operation on $G$. We say that
$*$ is a {\it quasigroup operation} if $*$ is left and right cancelable, that is, if

$$a*b=a*c \qquad\Longleftrightarrow\qquad b=c$$
and
$$b*a=c*a \qquad\Longleftrightarrow\qquad b=c,$$
respectively.

If, in addition, $G$
is a topological space and $*$ is
continuous, we say that $*$ is a
{\em topological quasigroup operation}. For the case when $G$ is finite quasigroup, the multiplication table for $*$ is a Latin square (that is, it has no repetition of elements on each row and on each column). Furthermore, it is easy to check that if for any
$g\in G$ it follows that $g*G=G*g=G$ (which always occurs if $G$ is finite), then $*$ is an associative quasigroup operation if and
only if $*$ is a group operation.

For a finite set $G$, we say that $s:G\to G$ is a cyclic permutation on $G$, if given any $x\in G$ we have that $G=\{s^k(x): k=0,\ldots, \# G-1\}$, where $\# G$ denotes the cardinality of $G$. The following lemma gives a sufficient and necessary condition on the cardinality of a finite set $G$ for the existence of quasigroup operations on $G$ for which a given cyclic permutation is an automorphism.

\begin{lem}\label{finiteQuasiGroup}
Given a finite set $G$ and a cyclic permutation $s:G\to G$, there exists a quasigroup operation $*$ on $G$ for which $s$ is an automorphism if and only if the cardinality of $G$ is odd.
\end{lem}

\begin{proof}
Let $n:=\# G$ and, without loss of generality, we can consider $G:=\{0,\ldots,n-1\}$ and the cyclic permutation on $G$ in the form $$s(x)=x\tilde +1,$$ where $\tilde +$ is the sum $mod\ n$.

If $n$ is odd, we can define $\lambda:=(n+1)/2\in\N$ and, since $gcd(\lambda,n)=1$ we have a quasigroup operation $*$ defined for all $x,y\in G$ by $$x*y:=\lambda (x\tilde + y),$$ where $\lambda z$ stands for $z$ summed $\lambda$ times with itself ($mod\ n$). Therefore, for any $x,y\in G$ we have that \begin{equation*}\begin{array}{lcl}s(x)*s(y)&=&(x\tilde+ 1)*(y\tilde+ 1)=\lambda[(x\tilde+ 1)\tilde+(y\tilde+ 1)]=\\\\&=&\lambda(x\tilde+ y\tilde+ 2)=\lambda (x\tilde+ y)\tilde+ \lambda 2=(x*y)\tilde+ 1=s(x*y).\end{array}\end{equation*}

Now, let us show that if $n$ is even, then it does not exist a quasigroup operation for which $s$ is automorphism. For this, consider that $*$ is some binary operation on $G$ for which $s$ is an automorphism. We can represent the action of $*$ on $G$ by a table where the entry in the row indexed by $x$ and column indexed by $y$ represents the product $x*y$. Note that since $s$ is an automorphism for $*$, then the table looks as follows:

\begin{center}
\begin{tabular}{|c|ccccc|}
  \hline
  $*$       & 0 & 1 & 2 & \ldots &  n-1 \\ \hline
  0      & $a_0$ & $a_1$ & $a_2$ & $\cdots$ &  $a_{n-1}$ \\
  1      & $a_{n-1}\tilde+ 1$ & $a_0\tilde+ 1$ & $a_1\tilde+ 1$ & $\ddots$ &  $a_{n-2}\tilde+ 1$ \\
  2      & $a_{n-2}\tilde+ 2$ & $a_{n-1}\tilde+ 2$ & $a_0\tilde+ 2$ & $\ddots$ & $a_{n-3}\tilde+ 2$  \\
  \vdots & $\vdots$ & $\ddots$ & $\ddots$ & $\ddots$ &  $\vdots$ \\
  n-1    & $a_1\tilde+ (n-1)$ & $a_2\tilde+ (n-1)$ & $a_3\tilde+ (n-1)$ & $\cdots$ & $a_0\tilde+ (n-1)$ \\
  \hline
\end{tabular}
\end{center}

\noindent where $a_0,a_1,\ldots,a_{n-1}\in\{0,\ldots,n-1\}$. Hence, for all $x,y\in G$, we can write $x*y=a_{y\tilde+(n-x)}\tilde+ x$. We recall that $*$ is a quasigroup operation if and only if the above table is a Latin square, that is, if and only if there is not repetition of elements in any row and any column of the table. Therefore, the sum $mod\ n$ of all elements of any row or of all elements of any column should result in the same value. But, if $n$ is even, supposing by contradiction that there is no repetitions in any row of the table (that is, $a_i\neq a_j$, for $i\neq j$), we get that the sum $mod\ n$ over any row is
$$\sum_{y=0}^{n-1}x*y=\sum_{y=0}^{n-1}(a_{y\tilde+(n-x)}\tilde+ x)=\frac{1}{2}(n-1)n\tilde+nx\ (mod\ n)=\frac{n}{2},$$
while the sum (mod $n$) over any column is
$$\sum_{x=0}^{n-1}x*y=\sum_{x=0}^{n-1}(a_{y\tilde+(n-x)}\tilde+ x)=\frac{1}{2}(n-1)n\tilde+\frac{1}{2}(n-1)n\ (mod\ n)=0,$$
which contradicts the assumption that the table is a Latin square.
\end{proof}

In the next section we shall look for quasigroup operations, for which a given dynamical system is an automorphism. In general, if $(X,T)$ is ergodic with nonzero entropy, then there does not exist such a quasigroup operation (with exception for dynamical systems on zero-dimensional spaces). In fact, to deal with the general case, we will need some `weakness' in the operation.

\begin{defn}\label{ergodic_quasigroup} Let $G$ be a topological space.
Given a probability measure $\mu$ on the Borelians of $G$, we will say that $*$ is a {\em weak quasigroup operation with respect to $\mu$} if $*$ is a quasigroup operation which is well defined for $\mu\times\mu$-almost all $(a,b)\in G\times G$. If in addition the operation $*$ is continuous on its
domain, then we will say it is a {\em topological weak quasigroup operation with respect to $\mu$}. When $*$ is a weak quasigroup operation with respect to $\mu$ on $G$, we will call $(G,*,\mu)$ a {\em (topological) weak quasigroup}. Furthermore, if a (topological) weak quasigroup is associative, then we will simply say it is a {\em (topological) weak group}
\end{defn}

Note that the definition of a weak quasigroup operation is made on the product space $G\times G$ and not on the space $G$. Thus, it is possible that there exist $x\in G$ and a non-null measure subset of $A\subset G$ such that for all $y\in A$ the products $x*y$ and $y*x$ are not defined.

On the other hand, if $*$ is a weak quasigroup operation with respect to some measure $\mu$ on $G$, then
given $x,y\in G$, the existence of the product $x*y$ does not imply the
existence of $y*x$ (except when $*$ is commutative). Furthermore, the cancelation property of a weak quasigroup operation $*$ means that if $x*y$ (or $y*x$) and $x*z$ (or $z*x$) are defined, then $x*y=x*z$ (or $y*x=z*x$) if and only if $y=z$. In the same way, the associativity of a weak group holds only if both $x*(y*z)$ and $(x*y)*z$ are defined.

\section{Weak quasigroups and expansive ergodically supported automorphisms}\label{main}

In order to construct a topological weak quasigroup for which a given topological dynamical system is an automorphism, we need the following results.


\begin{lem}\label{lemma1}
Let $\alpha:\mathcal{X}\to \mathcal{Y}$ be a continuous and onto map between topological spaces, and suppose $\mathcal{X}$ is compact and $\mathcal{Y}$ is Hausdorff. Given $\tilde{\mathcal{Y}}\subseteq \mathcal{Y}$, define $\tilde{\mathcal{X}}:=\alpha^{-1}(\tilde{\mathcal{Y}})$. If the restriction $\tilde\alpha:\tilde{\mathcal{X}}\to \tilde{\mathcal{Y}}$ is one-to-one, then $\tilde\alpha$ is a homeomorphism.
\end{lem}

 \begin{proof}
 We only need to check that $\tilde\alpha^{-1}$ is continuous.
 To achieve this, suppose  by contradiction, that $\tilde\alpha^{-1}$ is not continuous
 and let $(y_i)_{i\in\N}\in\tilde{\mathcal{Y}}$ be a sequence which converges to $y\in\tilde{\mathcal{Y}}$ but $\big(\tilde\alpha^{-1}(y_i)\big)_{i\in\N}$ does not converge to $\tilde\alpha^{-1}(y)$. It means that there should exist an open neighborhood of $\tilde\alpha^{-1}(y)$, $A\subset\tilde{\mathcal{X}}$, and a subsequence $(y_{i_j})_{j\in\N}$, such that \begin{equation}\label{contradic}\tilde\alpha^{-1}(y_{i_j})\notin A,\quad \forall j\in\N.\end{equation}

 Since $\mathcal{X}$ is compact, the sequence $(x_{i_j})_{j\in\N}\in\tilde{\mathcal{X}}$, where $x_{i_j}:=\tilde\alpha^{-1}(y_{i_j})$, has a subsequence which converges in $\mathcal{X}$. Let $(x_{i_{j_k}})_{k\in\N}$ be this subsequence and let $x\in\mathcal{X}$ be its limit. The continuity of $\alpha$ on $\mathcal{X}$ implies $y_{i_{j_k}}=\alpha\bigl(x_{i_{j_k}}\bigr)\to \alpha(x)$ as $k\to\infty$. But $y_{i_{j_k}}\to y$ as $k\to\infty$ and since $\mathcal{Y}$ is Hausdorff we get $\alpha(x)= y$. Thus, since $y\in\tilde{\mathcal{Y}}$, $\tilde{\mathcal{X}}=\alpha^{-1}(\tilde{\mathcal{Y}})$ and $\tilde\alpha$ is injective we get that $x$ is the unique element in the preimage of $y$ by $\alpha$, and therefore $x=\tilde\alpha^{-1}(y)$.
 Hence, $$\lim_{k\to\infty}\tilde\alpha^{-1}(y_{i_{j_k}})=\tilde\alpha^{-1}(y),$$ which is a contradiction with \eqref{contradic}. Thus we conclude that $\tilde\alpha^{-1}$ is continuous on $\tilde{\mathcal{Y}}$.

 \end{proof}

Hence, by using the above lemma we can prove that two invertible dynamical systems that are almost topologically conjugated are topologically conjugated outside of universally null sets:

\begin{theo}\label{almostconjugacy} If two invertible dynamical systems $(X,T)$ and $(Y,S)$ are almost topologically conjugated, then there exists a homeomorphism $\varphi:\tilde X\to \tilde Y$ between $\tilde X\subseteq X$ and $\tilde Y\subseteq Y$ total-measure subsets with respect to any ergodically supported measure, which is a topological conjugacy between $(\tilde X,T)$ and $(\tilde Y,S)$
\end{theo}

\begin{proof}

Let $(X^T,\s_{X^T})$ and $(Y^S,\s_{Y^S})$ be the inverse limit systems of $(X,T)$ and $(Y,S)$, respectively. Note that since both $(X,T)$ and $(Y,S)$ are invertible, then the projections $p_T:X^T\to X$ and $p_S:Y^S\to Y$ are homeomorphisms.

Let $(\Sigma,\s)$, $\f_T:\Sigma\to X^T$ and $\f_S:\Sigma\to Y^S$, be the Markov shift and the maps given in the definition of almost topological conjugacy. Also denote as $M_2\subseteq X^T$ and $P_2\subseteq Y^S$, and as $M_1:=\f_T^{-1}(M_2)$ and $P_1:=\f_S^{-1}(P_2)$, the universally null sets which make
the maps $\f_T:\Sigma\setminus M_1\to X^T\setminus M_2$ and $\f_S:\Sigma\setminus P_1\to Y^S\setminus P_2$ to be bijections. Denote as $\bar\f_T$ and  $\bar\f_S$ these restrictions of $\f_T$ to $\Sigma\setminus M_1$ and of $\f_S$ to $\Sigma\setminus P_1$, respectively.  From Lemma \ref{lemma1}, we get that $\bar\f_T$ and $\bar\f_S$ are homeomorphisms.

Note that since $p_T$ and $p_S$ are homeomorphisms, the sets $M_3:=p_T(M_2)$ and $P_3:=p_S(P_2)$ are also universally null sets. Denote as $\bar p_T$ and as $\bar p_S$ the restrictions $p_T:X^T\setminus M_2\to X\setminus M_3$ and $p_S:Y^S\setminus P_2\to Y\setminus P_3$.

Thus, the maps $\gamma_T:\Sigma\setminus M_1\to X\setminus M_3$ and $\gamma_S:\Sigma\setminus P_1\to Y\setminus P_3$ defined by $\gamma_T:=\bar p_T\circ\bar\f_T$ and $\gamma_S:=\bar p_S\circ\bar\f_S$ are also homeomorphisms.

Note that $\tilde \Sigma:=\Sigma\setminus(M_1\cup P_1)$, is a total-measure subset of $\Sigma\setminus M_1$ and of $\Sigma\setminus P_1$, with respect to any ergodically supported measure. Hence, $\tilde X:=\gamma_T(\tilde \Sigma)\subseteq X$ and $\tilde Y:=\gamma_S(\tilde \Sigma)\subseteq Y$ are  total-measure subsets with respect to any ergodically supported measure. Therefore we can consider $\tilde \gamma_T:\tilde\Sigma\to\tilde X$ and
$\tilde \gamma_S:\tilde\Sigma\to\tilde Y$ the restrictions of $\gamma_T$ and $\gamma_S$, respectively.

Finally, we define the homeomorphism $\varphi:\tilde X\to\tilde Y$ given by $$\varphi:= \tilde\gamma_S\circ\tilde\gamma_T^{-1}.$$
Since, all maps involved in the definition of $\varphi$ commute with the dynamical systems, we get that $\varphi$ is a topological conjugacy between $(\tilde X,T)$ and $(\tilde Y,S)$.

\end{proof}

Note that, since the sets $\tilde X$ and $\tilde Y$ are total-measure subsets for ergodically supported measures, then they
are dense in the interior of $X$ and $Y$, respectively. Thus, in general it is not possible to extend the map $\varphi$ to the the closure of $\tilde X$ and $\tilde Y$ (it only is possible in the particular case when $(\overline{int(X)},T)$ and $(\overline{int(Y)},S)$ are topologically conjugate).

\begin{cor}\label{entropia->conjugacao} Let $(X,T)$ and $(Y,S)$ be two topological
dynamical systems, and assume they are invertible, expansive, ergodically supported, have the
shadowing property, and have equal topological entropy and ergodic period. Then, there exist $\tilde X\subset X$ and $\tilde Y\subset Y$ total-measure subsets with respect to any ergodically supported measure such that $(\tilde X,T)$ and $(\tilde Y,S)$ are topologically conjugated.
\end{cor}

\begin{proof}
It is a consequence of Theorem \ref{SunMainTheo} and the previous theorem.
\end{proof}

Now, we are able to get sufficient conditions on a dynamical system that allow to define a topological weak quasigroup operation for which the map of the dynamical system becomes an automorphism.

\begin{theo}\label{entropia->operacao} Let $(X,T)$ be a topological
dynamical system, and assume it is invertible, expansive, ergodically supported with odd ergodic period, and has the
shadowing property. If $\mathbf{h}(T)=\log (N)$ for some positive integer
$N$, then there exists a topological weak quasigroup operation $\bullet$ with respect to any ergodically supported measure of $(X,T)$, for which $T$ is an automorphism.
\end{theo}

\begin{proof}

Let $B\in\N$ odd be the ergodic period of $(X,T)$. Define the dynamical system $(Y,S)$ as $Y:=\{0,1,\ldots, N-1\}^\Z\times\{0,1,\ldots, B-1\}$\sloppy\ and $S:=\s\times s$ the map where $\s:\{0,1,\ldots, N-1\}^\Z\to\{0,1,\ldots, N-1\}^\Z$ is the shift map and $s:\{0,1,\ldots, B-1\}\to\{0,1,\ldots, B-1\}$ is the cyclic permutation defined by $s(i):=i+1\ (mod\ B)$.

Note that since the $(\{0,1,\ldots, N-1\}^\Z,\s)$ is invertible, expansive, ergodically supported, has the
shadowing property, and has topological entropy $\mathbf{h}(\s)=\log (N)$, and since $(\{0,1,\ldots, B-1\},s)$ is a cyclic permutation of length $B$ and has zero entropy, then the product system $(Y,S)$ is invertible, expansive, ergodically supported, has the
shadowing property, has topological entropy $\mathbf{h}(S)=\log (N)$ and ergodic period $B$. Therefore, from Corollary \ref{entropia->conjugacao} there exist $\tilde X\subset X$ and $\tilde Y\subset Y$ total-measure subsets with respect to any ergodically supported measure and $\varphi:\tilde X\to \tilde Y$ which is a topological conjugacy between $(\tilde X,T)$ and $(\tilde Y,S)$.

Define on $Y$ the quasigroup operation $*$ given by
\begin{equation}\label{oper_canon}\Bigl((x_i)_{i\in\Z},a\Bigr)*\Bigl((y_i)_{i\in\Z}, b\Bigr):=\Bigl((x_i\tilde*y_i)_{i\in\Z},\lambda(a\tilde+b)\Bigr),\end{equation}
where $\tilde*$ is any quasigroup operation on $\{0,1,\ldots, N-1\}$, $\lambda:=(B+1)/2$ and $\tilde+$ is the sum $mod\ B$. It is straightforward that the shift map $\s$ is an automorphism for $\tilde*$ and, from Lemma \ref{finiteQuasiGroup}, the map $s$ is an automorphism for $\tilde+$. Thus, $S$ is an automorphism for $*$. Furthermore, since $\tilde*$ is a 1-block operation (see \cite{Sobottka2007}) and $\tilde+$ is continuous for the power set topology on $\{0,\ldots,B-1\}$, then $*$ is a topological quasigroup operation.

Denote by $\Theta:Y\times Y\to Y$ the map given by $\Theta(\x,\y)=\x*\y$, for any $\x,\y\in Y$.
Since $\Theta$ is continuous, the set $\Theta^{-1}(\tilde Y)\subseteq Y\times Y$ is a total-measure subset with respect to the product measure on $Y\times Y$ of any ergodically supported measure on $Y$.

Thus, $$\Lambda:=(\tilde Y\times\tilde Y)\cap\Theta^{-1}(\tilde Y)$$ is also a total-measure subset with respect to the product measure on $Y\times Y$ of any ergodically supported measure on $Y$. Note that, $\Lambda$ is the set of all pairs of points of $\tilde Y\times\tilde Y$ for what the product by
$*$ is a point lying in $\tilde Y$. Furthermore, since $\Theta$ commutes with the maps $S\times S$ and $S$, and $\tilde Y$ is $S$-invariant, we get that $\Lambda$ is $S\times S$-invariant.

Define $\Omega\subseteq \tilde X\times \tilde X$ by $$\Omega:=(\varphi\times\varphi)^{-1}(\Lambda).$$
Since $\Lambda$ is a total-measure subset with respect to the product measure on $Y\times Y$ of any ergodically supported measure on $Y$, and $\varphi\times\varphi:\tilde X\times\tilde X\to\tilde Y\times\tilde Y$ is a homeomorphism, then $\Omega$ is a total-measure subset with respect to the product measure on $X\times X$ of any ergodically supported measure on $X$. Hence, for any pair $(x,y)\in\Omega$ we can define the quasigroup operation $\bullet$ given by
\begin{equation}\label{bullet}x\bullet y:=\varphi^{-1}\bigl(\varphi(x)*\varphi(y)\bigr).\end{equation}

Note that $\bullet$ is well defined. In fact, since $(x,y)\in\Omega$, then $\bigl(\varphi(x),\varphi(y)\bigr)\in\Lambda$. Therefore $\bigl(\varphi(x)*\varphi(y)\bigr)\in\tilde Y$ and $\varphi^{-1}\bigl(\varphi(x)*\varphi(y)\bigr)\in\tilde X$.

Furthermore, for any $(x,y)\in\Omega$ it follows that

\small
$$\begin{array}{ll}
T(x\!\bullet\! y)&\!\!\!\!\!\!=\!\!T\Big(\varphi^{-1}\bigl(\varphi(x)\!*\!\varphi(y)\bigr)\Bigr)\!\!=\!\!\varphi^{-1}\Big(S\bigl(\varphi(x)\!*\!\varphi(y)\bigr)\Bigr)\\
&
\!\!\!\!\!\!=\!\!\varphi^{-1}\Big(S\bigl(\varphi(x)\bigr)\!*\!S\bigl(\varphi(y)\bigr)\Bigr)\!\!=\!\!\varphi^{-1}\Big(\varphi\bigl(T(x)\bigr)\!*\!\varphi\bigl(T(y)\bigr)\Bigr)\!\!=\!\!T(x)\!\bullet\! T(y).
\end{array}$$
\normalsize

\end{proof}




Note that the weak quasigroup operation $\bullet$ constructed in the proof of the previous theorem can be associative if, and only if, $(X,T)$ is ergodically aperiodic. In fact, if by contradiction we suppose $(X,T)$ is not ergodically aperiodic and $\bullet$ is associative, then $(Y,S)$ is also not ergodically aperiodic and the operaton $*$ will be a group operation. But $Y$ is an irreducible shift space, and from Theorem 1 (iv) in \cite{kitchens} it implies that $Y$ shall be topologically conjugate to a full shift and therefore it is ergodically aperiodic, a contradiction. Conversely, if $(X,T)$ is ergodically aperiodic, then we can take $Y$ as a full shift and define $*$ being a group operation, and thus $\bullet$ will be associative.

We remark that when $(X,T)$ has even ergodic period, due to Lemma \ref{finiteQuasiGroup}, if there exists a topological weak quasigroup operation for which $T$ is an automorphism, then the correspondent quasigroup operation on $Y$ cannot be construct as the product of a quasigroup operation on the full shift $\{0,1,\ldots, N-1\}^\Z$ with a quasigroup operation on the finite set $\{0,1,\ldots, B-1\}$, as we made in the proof of Theorem \ref{entropia->operacao}. Therefore, if exists, such quasigroup operation on $Y$ shall be a sliding block code $\Theta:Y\times Y\to Y$ of code size $k\geq 2$ (in the proof of Theorem \ref{entropia->operacao} $\Theta$ is a sliding block code of code size 1).

On the other hand, the condition on the topological entropy of $(X,T)$ is a sufficient condition which is used to allow us to construct a topological quasigroup operation on the shift space $Y$ for which $S$ is an automorphism. In fact, the results of \cite{kitchens} and \cite{Sobottka2007} state that we can define $k$-block group operations or 1-block quasigroup operations on $Y$, only if $\mathbf{h}(S)=\log(N)$ for some integer $N$. However, it is not known if there exists some $k$-block quasigroup operation, with $k\geq 2$, on a shift space with topological entropy $\log (\lambda)$ for a non-integer $\lambda$. In the same way, it is not clear if there is some restriction on the topological entropy of $(X,T)$ for the existence of a weak topological quasigroup operation on $X$ for which $T$ is an automorphism.

Observe that $(Y,S)$ has a unique maximum-entropy measure, which is also the unique maximum-entropy measure of $(\tilde{Y},S)$. Since $(\tilde{X},T)$ and $(\tilde{Y},S)$ are topologically conjugated, then $(\tilde{X},T)$ also has a unique maximum-entropy measure, which is also the unique maximum-entropy measure of $(X,T)$.

The next theorem gives sufficient conditions to decompose the dynamics of $(X,T)$ in terms of $T$-invariant weak subquasigroups.

\begin{theo}\label{standard_decomposition} Let $(X,T)$ be an invertible, expansive, ergodically supported and aperiodic topological dynamical system with the shadowing property. Suppose $\mathbf{h}(T)=\log(N)$, with $N\in\N$. If $p_1p_2\cdots p_q=N$\sloppy\ is a decomposition of $N$ into integers such that $p_i\geq 3$ for all $i=1,\ldots,q$, then there exists a topological weak quasigroup operation $\bullet$ on $X$ for which $T$ is an automorphism, and $T$-invariant weak subquasigroups $X_k\subseteq X$ for $k=1,\ldots,q$, such that almost all $x\in X$ with respect to the maximum-entropy measure of $(X,T)$ can be written as $x=x_1\bullet(x_2\bullet(\cdots(x_{q-2}\bullet(x_{q-1}\bullet x_q))))$, with $x_k\in X_k$.
\end{theo}

\begin{proof}

First, we need to define appropriately the topological quasigroup shift $Y$ in Theorem \ref{entropia->operacao}. For this, for each $k=1,\ldots,q$, consider the finite alphabet
$\Alf_k=\{1\ldots,p_k\}$ and the full shift over $p_k$ symbols
$Y_k:=\Alf_k^\Z$.

Since for any $k$ the alphabet $\Alf_k$ has cardinality $p_k\geq 3$, we can define on $\Alf_k$ an idempotent quasigroup operation, that is, a quasigroup operation $*_k$ such that for any $a\in\Alf_k$ it follows $a*_ka=a$ (see Table 1 of \cite{TeirlinckLindner}).

Therefore, define $$Y:=Y_1\times\cdots\times Y_q=\{(x_1^i,\ldots,x_q^i)_{i\in\Z}:\ (x_k^i)_{i\in\Z}\in Y_k,\ k=1,\ldots,q\},$$
and define on $Y$ the idempotent quasigroup operation $*$ given by
$$(x_1^i,\ldots,x_q^i)_{i\in\Z}*(y_1^i,\ldots,y_q^i)_{i\in\Z}=(x_1^i*_1y_1^i,\ldots,x_q^i*_qy_q^i)_{i\in\Z}.$$

Given $S\subset Y$ denote as $\langle S\rangle$ the subquasigroup of $(Y,*)$ generated by $S$, that is, the smallest subquasigroup which contains $S$. For $m\geq 1$ let $S^m$ be the set obtained by multiplying $S$ by itself $m$ times in any possible associative way. For example,
$$\begin{array}{lcl}
S^1&:=&S\\
S^2&:=&\displaystyle S*S\\
S^3&:=& S*(S*S)\cup (S*S)*S\\
&\vdots&
\end{array}$$

It is easy to check that $$\langle S\rangle=\bigcup_{m\geq 1}S^m.$$

Note that $Y$ is ergodically aperiodic and has topological entropy $\log(N)$. Now, let  $\tilde X\subset X$ and $\tilde Y\subset Y$ be the total-measure subsets with respect to any ergodically supported measure, and $\varphi:\tilde X\to \tilde Y$ be the topological conjugacy between $(\tilde X,T)$ and $(\tilde Y,\sigma)$, given by Theorem \ref{entropia->operacao}.

Let $\nu$ be the uniform Bernoulli measure on $Y$. We recall that $\nu$ is an ergodically supported measure and it is the maximum entropy measure for $(Y,\sigma)$. For each $k=1,\ldots,q$, denote as $\nu_k$ the projection of $\nu$ on the $k^{th}$ coordinate, that is, $\nu_k$ is the uniform Bernoulli measure on $Y_k$. In particular, $\nu=\bigotimes_{k=1}^q\nu_k=\nu_1\times\cdots\times\nu_q$.

Given $k=1,\ldots,q$, for each $j=1,\ldots,q$ and $j\neq k$, we can take $(z_{k,j}^i)_{i\in\Z}\in Y_j$ and define the section
\begin{equation}\label{section}S_k=\{(x_1^i,\ldots,x_q^i)_{i\in\Z}\in Y:\ x_j^i=z_{k,j}^i\ \forall j\neq k,\ \forall i\in\Z \}.\end{equation}
Note that, we can identify $S_k$ with $Y_k$ and without loss of generality we can consider the measure $\nu_k$ on $S_k$. Furthermore, since $*$ is idempotent, then for each $k=1,\ldots,q$, $S_k$ is a topological subquasigroup and $Y=S_1*(S_2*(\cdots(S_{q-2}*(S_{q-1}* S_q))))$.

On the other hand, since $\tilde{Y}$ has total measure, due to Fubini's Theorem we can get that for $\bigotimes_{j\neq k}\nu_j$-almost all choices of $(z_{k,j}^i)_{i\in\Z}\in Y_j$, $j\neq k$, we have $$\nu_k(S_k\cap\tilde{Y})=1.$$

Now, let $S_k$, for $k=1,\ldots,q$, be sections for which the above equality holds, and define $$\hat{S}_k:=\left\langle\bigcup_{n\in\Z}\sigma^n(S_k)\right\rangle.$$

We have that $\hat{S}_k$ is a $\sigma$-invariant subquasigroup of $Y$ and $\nu_k(\hat{S}_k\cap \tilde{Y})=1$ for each $k=1,\ldots,q$.
It means that $\nu$-almost all $\x\in \tilde{Y}$ can be written as $\x=\x_1*(\x_2*(\cdots(\x_{q-2}*(\x_{q-1}* \x_q))))$, with $\x_k\in \hat{S}_k$.

Finally, since the sets $\hat{S}_k\cap \tilde{Y}$ are $\sigma$-invariant and, since $\mu:=\nu\circ\varphi$ is the maximum-entropy measure for $(X,T)$, then the weak subquasigroups $X_k:=\varphi^{-1}(\hat{S}_k\cap \tilde{Y})$ satisfy the theorem.

\end{proof}

\begin{cor} Under the same hypotheses of Theorem \ref{standard_decomposition}, if $x\in X$ can be decomposed as $x=x_1\bullet(x_2\bullet(\cdots(x_{q-2}\bullet(x_{q-1}\bullet x_q))))$, with $\x_k\in X_k$, then
$T(x)=T(x_1)\bullet(T(x_2)\bullet(\cdots(T(x_{q-2})\bullet(T(x_{q-1})\bullet T(x_q)))))$.
\end{cor}

\cqd

Note that the main component of the previous theorem is to be able to: {\em (a)} define quasigroup operations on each $\mathcal{A}_k$ which allow to define the sections $S_k$ being subquasigroups; {\em (b)} assure that for each $k$ there exists a choice of $(z_{k,j}^i)_{i\in\Z}\in Y_j$, $j\neq k$, which makes $S_k$ be a total-measure set with respect to the measure $\nu_k$ (which is equivalent to say that $S_k$ is in the image of $\varphi$). To achieve {\em (a)} we could use any quasigroup operation with an idempotent element. However, to achieve {\em (b)} we need a complete freedom to choose the entries which define the section keeping the property that the section is a subquasigroup, and then we need that all the elements of the quasigroup shall be idempotent. The last imposes that each set $\mathcal{A}_k$ shall have cardinality greater than 2 (because any quasigroup operation on a set with two elements is isomorphic to $\Z_2$ which is not idempotent). As a consequence, if $(X,T)$ has topological entropy $\log(N)$ with $N\in\{2,8\}\cup\{2p:\ p\text{ is prime}\}$, we cannot use Theorem \ref{standard_decomposition} to decomposes its behavior (because any decomposition of $N$ will have 2 as a factor). In such a case if some decomposition is possible, it will depend on specific properties of the map $\varphi$ to find section $S_k$ which holds {\em (a)} and {\em (b)}. On the other hand, if
the dynamical system has topological entropy $\log(N)$ with $N\notin\{2,8\}\cup\{2p:\ p\text{ is prime}\}$, then we can always find a decomposition of its behavior into subquasigroups.\\

Note that if $Q_k$ and $R_k$ are two distinct sections given by \eqref{section}, then they are disjoint. Furthermore, $Q_k*R_k$ is also a section on the same coordinates.
Observe that any section $R_k$ is a universally null set. In fact, for any measure $\gamma$ on $Y$ it follows that $$\gamma\left(\bigcup_{n\in\Z}\sigma^n(R_k)\right)=\sum_{n\in I}\gamma(\sigma^n(R_k))=\sum_{n\in I}\gamma(R_k),$$
where $I$ is finite if the pairwise disjoint family $\{\sigma^n(R_k): n\in\Z\}$ is finite and $I=\Z$ otherwise. Hence, if by contradiction we suppose $\gamma$ is ergodically supported and $\gamma(R_k)>0$, then it follows that necessarily $I$ is finite. Therefore, $\bigcup_{n\in I}\sigma^n(R_k)\subsetneq Y$ is  closed and $\sigma$-invariant and, since $\gamma$ is ergodic and $\gamma(R_k)>0$, then $\gamma\left(\bigcup_{n\in I}\sigma^n(R_k)\right)=1$. But this implies that $Y\setminus \left(\bigcup_{n\in I}\sigma^n(R_k)\right)$ is a non-empty open set with null measure, which is a contradiction to the hypothesis that $\gamma$ is ergodically supported.

Now observe that the sets $\hat{S_k}$ in the proof of Theorem \ref{standard_decomposition} can be written as $$\hat{S_k}=\left\langle\bigcup_{n\in\Z}\sigma^n(S_k)\right\rangle=\bigcup_{m\geq 1}\left(\bigcup_{n\in\Z}\sigma^n(S_k)\right)^m,$$
where $\left(\bigcup_{n\in\Z}\sigma^n(S_k)\right)^m$ is the set obtained making all possible products between $m$ sections selected from $\bigcup_{n\in\Z}\sigma^n(S_k)$. Thus, for each $k$ and $m$ the set $\left(\bigcup_{n\in\Z}\sigma^n(S_k)\right)^m$ is a countable union of sections and therefore the subquasigroup $\hat{S_k}$ is also a countable union of sections and thus it is a universally null set. Consequently each weak subquasigroup $X_k$ is also a universally null sets

\section{Final discussion}

Note that the converse statement of Theorem \ref{entropia->operacao} would allow to extend the results of \cite{lind77} to the more general case where $T:X\to X$ is an expansive ergodically supported map with the shadowing property and an automorphism for some topological quasigroup operation. However, it does not seem to be direct that the existence of a topological (weak) quasigroup operation on $X$ for which $T$ is an automorphism implies that $\mathbf{h}(T)=\log(N)$. Note that, since any topological (weak) quasigroup operation on $X$ induces a topological weak quasigroup operation on the symbolic representation of $X$, then we should be able to assure that the existence of weak quasigroup operations on any Cartesian product of a shift space with a finite set implies that this shift space has topological entropy $\log(N)$. It would be some kind of extension of the consequences of Theorem 1 in \cite{kitchens} (for groups) and Theorem 4.25 in \cite{Sobottka2007} (for quasigroups).

It would also be interesting to study the case of non-invertible maps. In such a case we cannot apply Theorem \ref{entropia->conjugacao} since the projections $p_T:X^T\to X$ and $p_S:Y^S\to Y$ are not invertible. In fact, we only use the hypothesis that $T$ is invertible just to get $p_T:X^T\to X$ and $p_S:Y^S\to Y$ being homeomorphism and thus to assure that $\f_T$ and $\f_S$ are bimeasurable functions in Theorem \ref{almostconjugacy} which allowed to construct the topological conjugacy $\varphi:\tilde{X}\to\tilde{Y}$. However, there are several examples of non-invertible maps which are topologically conjugate to shift spaces outside of a universally null set and thus we can apply Theorem \ref{entropia->operacao} to construct an algebraic operation for which $T$ is an automorphism without the use of Theorem \ref{almostconjugacy} (for instance, the maps on the unit interval with the form $T(x)=Mx\ (mod\ M)$ are topologically conjugate to $\{0,\ldots,M-1\}^\N$ outside of the sets $\{i/M^k:\ k\geq 1,\ 0\leq i\leq M^k-1\}\subset [0,1]$ and $\{(x_i)_{i\in\N}:\ \sum_{i\in\N}x_i<\infty\}\subset \{0,\ldots,M-1\}^\N$).

\section*{Acknowledgment}

This work was supported
by National Counsel of Technological and Scientific Development-Brazil grants 304813/2012-5 and 304457/2009-4. The author was partially supported by the project DIUC 207.013.030-1.0 (UDEC-Chile).


\end{document}